\documentclass{article}
\usepackage[T1]{fontenc}
\usepackage{amsmath,amssymb,verbatim,amsthm,amscd}
\usepackage[all]{xy}
\usepackage{titlesec}
\usepackage[margin=22.5truemm]{geometry}
\usepackage{indentfirst}
\usepackage[utf8]{inputenc}
\usepackage{tikz-cd}

\newcommand{\C}{\mathbb{C}}

\newcommand{\cC}{\mathcal{C}}

\newcommand{\cI}{\mathcal{I}}

\newcommand{\cO}{\mathcal{O}}

\renewcommand{\i}{\sqrt{-1}}
\renewcommand{\leq}{\leqslant}
\renewcommand{\geq}{\geqslant}
\newcommand{\abs}[1]{\left\lvert#1\right\rvert}

\DeclareMathOperator{\imaginary}{\sqrt{-1}}

\DeclareMathOperator{\codim}{codim}

\DeclareMathOperator{\dd}{\sqrt{-1}\partial\bar{\partial}}
\DeclareMathOperator{\dbar}{\bar{\partial}}

\numberwithin{equation}{section}

\newtheorem{prop} {Proposition} [section]
\newtheorem{thm}[prop] {Theorem}
\newtheorem{main thm}[prop] {Main Theorem}

\newtheorem{dfn}[prop] {Definition}
\newtheorem{lem}[prop] {Lemma}

\newtheorem{exam}[prop]{Example}

\newtheorem*{thmA*}{\bf{Theorem A}}
\newtheorem*{thmB*}{\bf{Theorem B}}
\newtheorem*{thmC*}{\bf{Theorem C}}
\newtheorem*{thmD*}{\bf{Theorem D}}
\newtheorem*{thmE*}{\bf{Theorem E}}

\newtheorem{dfn*}{\bf{Definition}}

\begin{document}

\title{\Large{A Generalized $L^2$ Division Theorem and the Strong Openness Property of Multiplier Ideal Sheaves}}
\author{Masakazu Takakura}
\date{}

\maketitle
\begin{abstract}
In this paper, we generalize Guan's sharp effective strong openness theorem to the case of several plurisubharmonic weights.
The proof uses a generalized Skoda-type $L^2$ division theorem for multiplier ideal sheaves. We also discuss the relationships among the $L^2$-division theorem, the $L^2$-extension theorem and the effective openness theorem.
\end{abstract}

\tableofcontents
\section{Introduction}

Let $\Omega$ be a domain in $\mathbb{C}^n$, and let $\varphi$ be a plurisubharmonic function on $\Omega$.
The multiplier ideal sheaf $\mathcal{I}(\varphi)$ associated with $\varphi$ is defined by
$
\mathcal{I}_x(\varphi)
:=
\{\, f\in\mathcal{O}_x \; ; \; |f|^2 e^{-\varphi} \text{ is integrable in a neighborhood of } x \,\}.
$
The multiplier ideal sheaf is a fundamental object in complex geometry (see \cite{Demailly lec}).
In \cite{DK}, Demailly and Koll\'ar formulated the openness conjecture for complex singularity exponents.

\medskip

\noindent
\textbf{Strong openness conjecture.}
For every $x\in\Omega$, there exists $\varepsilon>0$ such that
$
\mathcal{I}_x(\varphi)=\mathcal{I}_x\bigl((1+\varepsilon)\varphi\bigr).
$

\medskip

\noindent
\textbf{Weak openness conjecture.}
If $\mathcal{I}_x(\varphi)=\mathcal{O}_x$, then there exists $\varepsilon>0$ such that
$
\mathcal{I}_x\bigl((1+\varepsilon)\varphi\bigr)=\mathcal{O}_x .
$

\medskip
For $n=2$, Favre--Jonsson proved the weak openness conjecture (\cite{FJ}). In \cite{B openness}, Berndtsson later proved the weak openness conjecture in arbitrary dimension, after treating the case of isolated singularities in \cite{B projective}.
Finally, Guan--Zhou proved the strong openness theorem (\cite{GZ openness}).

The next natural question is how large the constant $\varepsilon>0$ can be chosen. In \cite{GZ sharp eff}, Guan obtained the following sharp effective result (Guan--Zhou obtained an earlier non-sharp result in \cite{GZ eff}):

Let $\Omega$ be a pseudoconvex domain in $\mathbb{C}^n$ containing $0$.
Let $\varphi$ be a negative plurisubharmonic function.
Let $f$ be a holomorphic function satisfying
$
I=\int_{\Omega} |f|^2 e^{-\varphi} < \infty .
$
For $p>1$, define
$
c_{\varphi,f}(p)
=
\inf \left\{
\int_{\Omega} |g|^2 \; ;\;
f-g \in \mathcal{I}_0(p\varphi)
\right\}.
$

\begin{thm}\label{Guan's sharp effective openness theorem}
The inequality
$
c_{\varphi,f}(p) \le (1-1/p) I
$ holds.
\end{thm}
In the course of this proof, Guan established the concavity property of the $L^2$ minimal integral associated with plurisubharmonic functions. As an application, he provided a proof of the Saitoh conjecture (\cite{saitoh}), namely, a comparison theorem between the Bergman kernel and the Szeg\"{o} kernel.

The main result of this paper is an effective strong openness theorem for several plurisubharmonic functions: Let $\Omega$ be a pseudoconvex domain in $\mathbb{C}^n$, $\psi$ be a plurisubharmonic function and $(\varphi_1,\dots,\varphi_r)$ be a tuple of negative plurisubharmonic functions on $\Omega$. Set $m:=\min(n,r)$. For any holomorphic function $f \in \cO(\Omega)$ satisfying
$I = \int_{\Omega} \abs{f}^2 e^{-(\psi+\varphi_1+\cdots+\varphi_r)} < +\infty$ and any positive numbers $p_1,\dots,p_r > 1$ satisfying $\sum_{i=1}^r 1/p_i < m$,
we put $$ c_{\varphi_1,\dots,\varphi_r,f,\psi}(p_1,\dots,p_r) := \inf \big\{ \int_{\Omega} \abs{g}^2 e^{-\psi} : g \in \cO(\Omega)\;, \;f - g \in \sum_{i=1}^r \Gamma(\Omega, \cI(p_i\varphi_i + \psi)) \big\}\;\;\;.$$
We show the following.

\begin{thm}\label{main theorem}
  The inequality
  $c_{\varphi_1,\dots,\varphi_r,f,\psi}(p_1,\dots,p_r) \le \left(1- \frac{1}{m}\sum_{i=1}^r \frac{1}{p_i}\right)I$
  holds.
\end{thm}
We first make a minor remark concerning the definition of the minimal integral. In Guan's formulation, the congruence condition is stated at the level of the germ at the origin. More precisely, the admissible holomorphic functions satisfy $(f-g,0)\in \mathcal I(p\varphi)_0$. In Theorem~\ref{main theorem}, by contrast, the admissible function is required to satisfy the global condition
\[
f-g\in \sum_{i=1}^r
\Gamma\bigl(\Omega,\mathcal I(p_i\varphi_i+\psi)\bigr).
\]
However, Guan's proof also yields a function satisfying this global condition. Therefore, the global formulation in Theorem~\ref{main theorem} is useful for our purposes, but it should not be regarded as a new point.

To prove the main theorem, we use the following generalized $L^2$ division theorem for multiplier ideal sheaves.
\begin{thm}[cf. Theorem 6.1 in \cite{M.T}]\label{sharp L^2-division theorem}
  Let $\Omega$ be a pseudoconvex domain in $\mathbb{C}^n$, $\psi$ be a  plurisubharmonic function and $(\varphi_1,\dots,\varphi_r)$ be a tuple of plurisubharmonic functions on $\Omega$ satisfying $\sum_{i=1}^r e^{\varphi_i} <1.$ We set $q = \min(n,r-1)$. For any holomorphic function $f \in \cO(\Omega)$ such that
  $\int_{\Omega} \frac{\abs{f}^2 e^{-\psi}}{(\sum_{i=1}^r e^{\varphi_i})^{q+1}} < +\infty,$
  there exist holomorphic functions $(h_1,\dots,h_r)$ such that
  $$ \sum_{i=1}^r h_i =f\;\;\;\text{and}$$
  $$\int_{\Omega} \sum_{i=1}^{r}\abs{h_i}^2 e^{-\varphi_i-\psi} \le \int_{\Omega} \frac{\abs{f}^2 e^{-\psi}}{(\sum_{i=1}^r e^{\varphi_i})^{q+1}}.$$
  This estimate implies the following inclusion of multiplier ideal sheaves:
$
\mathcal{I}\left( (q+1) \log \left( \sum_i e^{\varphi_i} \right) \right) \subset \sum_{i=1}^r \mathcal{I}(\varphi_i).
$
  \end{thm}
  When $\varphi$ admits a decomposition $\varphi=\sum \varphi_i$, Theorem~\ref{main theorem} yields a more refined estimate than Theorem~\ref{Guan's sharp effective openness theorem}. In particular, it suggests the possibility of a more precise analysis of the plurisubharmonic functions. Moreover, this approach has the advantage of simplicity of its proof, as it avoids the technical arguments involving  the construction of cut-off functions that are commonly used in proofs of the openness conjecture.

The proof of Theorem \ref{main theorem} is inspired by the arguments in \cite{A25}. In  \cite{A25}, Albesiano gave a new proof of the $L^2$-extension theorem using the $L^2$-division theorem. From this perspective, one may conclude that the $L^2$-division theorem is a comprehensive result leading to a wide range of applications.

In Section \ref{relation}, we will discuss the relationships among the $L^2$-division theorem, the $L^2$-extension theorem, and the strong openness theorem.

\section{Sharp $L^2$-division theorem for multiplier ideal sheaves}\label{Sharp L^2-division theorem for multiplier ideal sheaves}
The aim of this section is to recall the statement and the proof of the  generalized $L^2$
 division theorem (= Theorem \ref{sharp L^2-division theorem}) established in \cite{M.T}.
A proof of this result was given in \cite{M.T}, but we present a slightly simplified argument in this section.

\subsection{Nakano positivity of Chern curvature}
We begin by recalling the basic notion of curvature positivity, namely the Nakano positivity of Chern curvature.
Let \( E \to X \) be a holomorphic vector bundle of rank \( r \) over an \( n \)-dimensional complex manifold \( X \), and let \( h \) be a smooth Hermitian metric on \( E \). The Chern curvature of \( h \) is given by
\[
\Theta_h = \bar{\partial}(\partial h \cdot h^{-1}),
\]
which is a \((1,1)\)-form with values in \( \mathrm{End}(E) \).

\begin{dfn}
We say that \( h \) is \emph{Nakano (semi)-positive} if the associated Hermitian form on \( E \otimes T_X \) defined by \( h \otimes (\sqrt{-1}\Theta_h) \) is (semi)-positive definite.
\end{dfn}

We now explain how Nakano positivity may be transferred to a subbundle or a quotient bundle. Let \( S \subset E \) be a holomorphic subbundle of rank \( r_S \), and let \( Q = E/S \) be the quotient bundle. The metric \( h \) induces Hermitian metrics \( h_S \) and \( h_Q \) on \( S \) and \( Q \), respectively. Let \( q = \min(n, r_S) \).

\begin{thm}[Theorem 11.18 in (\cite{Demailly lec})]\label{Nakano positivity of sub bundle and quotient bundle}
If $h$ is Nakano semi-positive, then the
Hermitian metrics on $\det Q$ and \( S \otimes \det(Q)^{\otimes q} \) induced by \(h_Q\) and \(h_S\), respectively, are Nakano semi-positive.
\end{thm}
Next, we present a key inequality that plays an essential role in solving the \(\bar{\partial}\)-equation with $L^2$-estimates in a subbundle.
Let \( f \) be a holomorphic section of \( Q \). Then, using the following commutative diagram, we can define a (not necessarily holomorphic) section \( g^* f \) of \( E \):

\[
\begin{CD}
    E @>{g}>> Q\\
    @AAA @VVV\\
    E^\vee @<{g^\vee}<< Q^\vee
\end{CD}\]

Here, the vertical maps are the conjugate-linear Riesz isomorphisms induced by the Hermitian metrics \( h \) and \( h_Q \), respectively. At this point, we observe that the \((0,1)\)-form \(\bar{\partial} g^* f\) takes values in the subbundle \(S \subset E\) and is \(\bar{\partial}\)-closed.
This is a consequence of the holomorphicity of the quotient map \( g : E \to Q \) and the chain rule. Indeed, since \( f \in \Gamma(X, Q) \) is a holomorphic section, we have:
\[
0 = \bar{\partial} f = \bar{\partial}(g \circ g^* f) = g (\bar{\partial}(g^* f)).
\]

Now, we assume that the Hermitian metric \( h \) is Nakano positive.
We further observe that the Chern curvature form \( \sqrt{-1}\Theta_{\det Q} \) is a positive \((1,1)\)-form. Therefore, it defines a Hermitian metric on the holomorphic tangent bundle \( T_X \).
For $f \in \Gamma(X,Q)$ and $q = \min(n,r_S)$, we will use the following inequality.

\begin{lem}[Theorem 12.3 in \cite{Demailly lec}]\label{Skoda trick}
Let \( g^* f \in \Gamma(X, E) \) be the section associated with a holomorphic section \( f \in \Gamma(X, Q) \). Then the following inequality holds:
\[
| \bar{\partial} g^* f |^2_{h_S, \sqrt{-1}\Theta_{\det Q}} \leq q \, | g^* f |^2,
\]
where the norm on the left-hand side is taken with respect to the Hermitian metric introduced by $h_S$ and $\sqrt{-1}\Theta_{\det Q}$.
\end{lem}
\subsection{Twisted $L^2$-existence theorem}
In this subsection, we introduce the optimal twisted \(L^2\)-existence theorem for \(\bar{\partial}\)-equations. This theorem provides a sharp estimate under Nakano positivity assumptions and will serve as a fundamental tool in the proof of Theorem \ref{sharp L^2-division theorem}.

First, we recall the following twisted version of Hörmander’s \( L^2 \) existence theorem.
\begin{lem}[Theorem 13.4 in \cite{Demailly lec}]\label{twisted L^2-existence theorem} Let $\Omega \subset \C^n$ be a pseudoconvex domain, $(E,h_E)$ be a Hermitian vector bundle of rank $r_E$, $\omega$ be a continuous Hermitian metric. Let $\eta$ be a positive $\cC^2$ function and $\lambda$ be a continuous positive function. Assume that for the operator $B = \i \eta \Theta_E - (\dd \eta \otimes Id_E +\partial \eta \wedge \dbar \eta/\lambda \otimes Id_E)$,  the induced Hermitian form $h_E \otimes B - h_E \otimes \omega$ is positive definite on $E\otimes T\Omega$. Then for any $E$-valued $\dbar$-closed $(0,1)$-form $\beta$ such that
  \begin{equation}
      \int_{\Omega} \abs{\beta}^2_{h_E,\omega}  < +\infty,
  \end{equation}
there exists an $E$-valued measurable section $u$ such that
\begin{gather}
  \dbar u = \beta\;\text{in the sense of distributions}\\
  \int_{\Omega} \frac{\abs{u}_{h_E} ^2} {(\eta + \lambda)} \le \int_{\Omega} \abs{\beta}^2_{h_E,\omega}.
\end{gather}

\end{lem}
By choosing appropriate weights \( h_E \), \( \eta \), and \( \lambda \), we obtain the following:

\begin{thm}\label{a Priori}
  Let $\Omega \subset \C^n$ be a pseudoconvex domain and let $(E,h_E)$ be a rank-$r_E$ Nakano semi-positive bundle. Let $\varphi$ be a negative $C^2$ plurisubharmonic function on $\Omega$. Then for any $E$-valued $\dbar$-closed $(0,1)$-form $\beta$ and positive number $\alpha>0$ such that
  \begin{equation}
      \int_{\Omega} \frac{1-e^{\alpha\varphi}}{\alpha} \abs{\beta}^2_{h_E,\dd\varphi} < +\infty,
  \end{equation}
there exists an $E$-valued measurable section $u$ such that
\begin{gather}
  \dbar u = \beta \; \text{in the sense of distributions}\\
  \int_{\Omega} e^{\alpha\varphi} \abs{u}_{h_E}^2 \le \int_{\Omega} \frac{1-e^{\alpha\varphi}}{\alpha} \abs{\beta}^2_{h_E,\dd\varphi}.
\end{gather}
\end{thm}
\begin{proof}[Proof of Theorem \ref{a Priori}]
  We consider the following smooth positive functions defined on \(\mathbb{R}_{<0}\):

\[
T(x) := -\log \int_x^0 e^{\alpha y} dy= -\log\left(\frac{1-e^{\alpha x}}{\alpha}\right) \,, \quad
U(x) := \int_x^0\int_y^0 e^{\alpha t} \,dt\,dy.
\]

We define the weight function by
$
\eta(x) := U(x) e^{T(x)}
$
and $
\lambda(x) := -\frac{\eta'(x)}{T'(x)}.
$
These functions satisfy the following ODE system:

$$\frac{e^{-T}}{\eta + \lambda} = e^{\alpha x}$$
$$\eta T' -\eta' = 1$$
$$\eta T'' - \eta'' - \frac{(\eta')^2}{\lambda} = 0.$$

For the weights $ \eta(\varphi), \lambda(\varphi)$ and the metric $\tilde{h}_E = e^{-T(\varphi)}h_E$, we can compute the twisted curvature operator as follows:
\begin{gather}
B := \imaginary \eta \, \Theta_{\tilde{h}_E} - \imaginary \partial \bar{\partial} \eta \otimes \mathrm{Id}_E
- \frac{\imaginary \partial \eta \wedge \bar{\partial} \eta}{\lambda} \otimes \mathrm{Id}_E\\
= \imaginary\eta(\varphi) \, \Theta_{h_E} + \left[ \eta(\varphi) T'(\varphi) - \eta'(\varphi) \right] \imaginary \partial \bar{\partial} \varphi \otimes \mathrm{Id}_E + \imaginary \left[\eta(\varphi) T''(\varphi) - \eta''(\varphi) - \frac{(\eta'(\varphi))^2}{\lambda(\varphi)}\right] \partial \varphi \wedge \dbar \varphi\\
=\imaginary\eta \, \Theta_{h_E} +\dd\varphi\otimes Id_E\\
\ge \dd\varphi\otimes Id_E.
\end{gather}
Then we may apply Theorem \ref{twisted L^2-existence theorem} to conclude the desired \(L^2\)-estimate.

\end{proof}
\begin{exam}\label{ex:equality-refined-hormander}
The refined H\"ormander estimate is sharp. Let $\Omega=\Delta$ be the unit disc and set $\varphi(z)=|z|^2-1$. We take $C(t)=e^t$ and $D(t)=1-e^t$. Then $D'(t)=-C(t)$, and $\varphi|_{\partial\Delta}=D(\varphi)|_{\partial\Delta}=0$.

Consider the equation $\bar\partial u=d\bar z$. Its $L^2(\Delta,C(\varphi))$-minimal solution is $u=\bar z$. Indeed, every solution is of the form $u=\bar z+h$, where $h\in\mathcal O(\Delta)$, and $\bar z$ is orthogonal to $\mathcal O(\Delta)$ with respect to the radial weight $C(\varphi)=e^{|z|^2-1}$.

We now verify the equality case without evaluating either integral explicitly. Since $\bar\partial\varphi=z\,d\bar z$ and $D'=-C$, we have
\[
\begin{aligned}
\bar\partial\bigl(D(\varphi)\bar z\,dz\bigr)
&=
\bar\partial\bigl(D(\varphi)\bar z\bigr)\wedge dz\\
&=
\left(D'(\varphi)|z|^2+D(\varphi)\right)d\bar z\wedge dz\\
&=
\left(C(\varphi)|z|^2-D(\varphi)\right)dz\wedge d\bar z.
\end{aligned}
\]
Therefore, using $d\lambda=\frac{\sqrt{-1}}{2}dz\wedge d\bar z$, we obtain
\[
\begin{aligned}
\int_\Delta C(\varphi)|u|^2\,d\lambda
&=
\int_\Delta C(\varphi)|z|^2\,d\lambda\\
&=
\int_\Delta D(\varphi)\,d\lambda
+
\frac{\sqrt{-1}}{2}
\int_\Delta
\bar\partial\bigl(D(\varphi)\bar z\,dz\bigr)\\
&=
\int_\Delta D(\varphi)\,d\lambda
+
\frac{\sqrt{-1}}{2}
\int_{\partial\Delta}D(\varphi)\bar z\,dz\\
&=
\int_\Delta D(\varphi)\,d\lambda.
\end{aligned}
\]
Here, the boundary term vanishes because $D(\varphi)=0$ on $\partial\Delta$. Thus, equality holds in the refined H\"ormander estimate.

This example indicates that the boundary condition $\varphi=0$ is essential for the equality case.
\end{exam}

\subsection[A proof of sharp $L^2$-division theorem for multiplier ideal sheaves]{A proof of sharp $L^2$-division theorem for multiplier ideal sheaves}
\begin{proof}
  The proof proceeds in two steps. The case $r=1$ is immediate, so we assume $r\geq2$.
  \subsubsection*{Step 1. Smooth case}
  We first establish the result under the assumption that the weight functions \(\psi, \varphi_1, \dots, \varphi_r\) are smooth. We define
$
E := \Omega \times \mathbb{C}^r$, $Q := \Omega \times \mathbb{C}
$
to be the trivial holomorphic vector bundles of rank \(r\) and \(1\), respectively.
Define a holomorphic bundle morphism \( g: E \to Q \) by
$
g := (1, 1, \dots, 1).
$
Then the kernel \( S := \ker g \subset E \) defines a holomorphic subbundle of rank \( r - 1 \).

We equip \( E \) with the Hermitian metric \( h \) given by the diagonal matrix
$
h := \mathrm{diag}(e^{-\varphi_1}, \dots, e^{-\varphi_r})
$.
Since the curvature of each line bundle component is given by
$
\sqrt{-1}\Theta_{e^{-\varphi_j}} = \sqrt{-1} \, \partial \bar{\partial} \varphi_j \geq 0,
$
the total metric \( h \) on \( E \) is Nakano semi-positive.
The Hermitian metric on the quotient bundle \( Q \) induced by \( h \) is given by
$
h_Q = e^{-\varphi},
$
here we define the function
$
\varphi := \log \left( e^{\varphi_1} + \cdots + e^{\varphi_r} \right).
$
By Theorem \ref{Nakano positivity of sub bundle and quotient bundle}, the metric
$
 \tilde{h}_S = e^{-q \varphi} h_S
$
on the subbundle \( S \) is Nakano semi-positive, where \( q = \min(n, r - 1) \).

Let \( f \in \Gamma(\Omega, Q) \) be a holomorphic section. Using the lifting \( g^* f  \in C^{\infty}(\Omega, E) \), we consider the \(\bar{\partial}\)-equation
$
\bar{\partial} u = \bar{\partial} g^* f,
$
to be solved with values in the subbundle \( S \subset E \).
We apply Theorem \ref{a Priori} and Lemma \ref{Skoda trick} with twisting weight \( \alpha = q \), and conclude that there exists a solution $u$ for $\dbar$-equation $
\bar{\partial} u = \bar{\partial} g^* f
$ such that

\begin{align*}
\int_\Omega |u|^2_{h_S} e^{-\psi} =
\int_\Omega |u|^2_{\tilde{h}_S} e^{q \varphi -\psi} &\leq \int_\Omega \frac{1-e^{q\varphi}}{q} |\bar{\partial} g^* f|^2_{h_E, \sqrt{-1}\Theta_{\det Q}} e^{-q\varphi -\psi}\\ &\leq
\int_\Omega (1-e^{q\varphi})|g^* f|^2_{h_E} e^{-q\varphi-\psi}
\end{align*}

Define \( F := (F_1, \dots, F_r) := g^* f - u \).
Then \( F \in \Gamma(\Omega, E) \) is a holomorphic section satisfying
$
g \cdot F = g \cdot (g^* f - u) = f,
$
since \( g^* f \) maps to \( f \) and \( u \in \ker g \). Thus, the components \( F_1, \dots, F_r \) form a tuple of holomorphic functions on \( \Omega \) such that
$
\sum_{j=1}^r F_j = f.
$

Moreover, by construction, \(g^*f\in S^\perp\) and \(u\in S\), so they are orthogonal with respect to the Hermitian metric \(h\). Therefore, we have the orthogonal decomposition
\[
  \|F\|^2_{L^2}= \|g^* f\|^2_{L^2}+ \|u\|^2_{L^2}.
\]

Using the \(L^2\)-estimate from the previous step, we obtain:
\[
\int_\Omega |F|^2_h e^{-\psi} = \int_\Omega |g^* f|^2_h e^{-\psi}+ \int_\Omega |u|^2_h e^{-\psi} \leq \int_\Omega |g^* f|^2_h e^{-q\varphi-\psi} \leq \int_\Omega |f|^2 e^{-(q+1)\varphi-\psi}.
\]

which is the desired \(L^2\)-estimate.
\subsubsection*{Step 2. General case}
We now consider the general case, where the plurisubharmonic weights \( \varphi_1, \dots, \varphi_r \) and \( \psi \) may be singular.

Since \( \Omega \subset \mathbb{C}^n \) is a pseudoconvex domain,
there exists an increasing sequence of relatively compact pseudoconvex subdomains \( \Omega_k \Subset \Omega \) such that
$
\bigcup_{k} \Omega_k = \Omega.
$

Moreover, for each \( i = 1, \dots, r \), we can find a sequence of smooth plurisubharmonic functions
$
\varphi_i^{(k)} \in C^\infty(\Omega_k)
$
such that \( \varphi_i^{(k)} \searrow \varphi_i \) pointwise on compact subsets of \( \Omega \) as \( k \to \infty \). Similarly, we approximate the weight \( \psi \) by smooth plurisubharmonic functions \( \psi^{(k)} \). We choose the regularizations so that $\sum_{i=1}^r e^{\varphi_i^{(k)}}<1$ on $\Omega_k$.

We now apply Step 1 (i.e., the smooth case of the \(L^2\)-division theorem) to each tuple of smooth weights \( (\varphi_1^{(k)}, \dots, \varphi_r^{(k)}, \psi^{(k)}) \) on the domain \( \Omega_k \). Then, for each \( k \), we obtain holomorphic functions \(F_1^{(k)}, \dots, F_r^{(k)} \in \mathcal{O}(\Omega_k)\) satisfying:
\[
\sum_{j=1}^r F_j^{(k)} = f,
\]
and the following estimate:
\[
\int_{\Omega_k} \sum_{j=1}^r |F_j^{(k)}|^2 e^{-\varphi_j^{(k)} - \psi^{(k)}} \leq \int_\Omega \frac{|f|^2}{(\sum e^{\varphi_i})^{q+1}} e^{-\psi}.
\]

Passing to a subsequence and using standard weak compactness in \(L^2_{\mathrm{loc}}\), we obtain a limiting tuple of holomorphic functions \( (F_1, \dots, F_r) \in \mathcal{O}(\Omega)^r \) satisfying
\[
\sum_{j=1}^r F_j = f, \quad \text{and} \quad \int_{\Omega} \sum_{j=1}^r |F_j|^2 e^{-\varphi_j - \psi} \leq \int_\Omega \frac{|f|^2}{(\sum e^{\varphi_i})^{q+1}} e^{-\psi}.
\]

This completes the proof in the general (possibly singular) case.

\end{proof}

\section{Proof of the main result}

The purpose of this section is to prove Theorem \ref{main theorem}.
\begin{lem}
Let
$
\theta_k(p_1,\dots,p_r)
:=
\int_{\mathbb{R}_{\ge 0}^r}
\left(1+x_1^{p_1}+\cdots+x_r^{p_r}\right)^{-k}
\,dx_1\cdots dx_r,
$
where $p_1,\dots,p_r>0$. Assume that
$
k>\sum_{i=1}^r \frac{1}{p_i}.
$
Then we have
$$
\frac{\theta_{k+1}(p_1,\dots,p_r)}{\theta_k(p_1,\dots,p_r)}
=
\frac{
k-\sum_{i=1}^r \frac{1}{p_i}
}{k}.
$$
\end{lem}

\begin{proof}
By the change of variables $y_i=x_i^{p_i}$, we have

$$
\theta_k(p_1,\dots,p_r)
=
\frac{1}{p_1\cdots p_r}
\int_{\mathbb{R}_{\ge 0}^r}
(1+y_1+\cdots+y_r)^{-k}
\prod_{i=1}^r y_i^{a_i-1}
\,dy_1\cdots dy_r.
$$

Set
$
a_i:=\frac{1}{p_i}
\;(i=1,\dots,r).
$ We now evaluate the latter integral. Write
$
s:=y_1+\cdots+y_r,
$,
$
t_i:=\frac{y_i}{s}
$
so that $t_i\ge 0$, $\sum_{i=1}^r t_i=1$, and $y_i=s t_i$. The Jacobian of this transformation is
$
dy_1\cdots dy_r
=
s^{r-1}\,ds\,dt_1\cdots dt_{r-1},
$
where $(t_1,\dots,t_{r-1})$ ranges over the standard simplex
$
\Delta_{r-1}
:=
\left\{
(t_1,\dots,t_r)\in \mathbb{R}_{\ge 0}^r
\; ; \;
t_1+\cdots+t_r=1
\right\}.
$
Hence
$
\prod_{i=1}^r y_i^{a_i-1}
=
s^{\sum_{i=1}^r a_i-r}
\prod_{i=1}^r t_i^{a_i-1},
$
and thus
$$
\theta_k(p_1,\dots,p_r)
=
\frac{1}{p_1\cdots p_r}
\left(
\int_0^\infty
(1+s)^{-k}
s^{\sum_{i=1}^r a_i-1}
\,ds
\right)
\left(
\int_{\Delta_{r-1}}
\prod_{i=1}^r t_i^{a_i-1}
\,dt
\right).
$$

The first integral is the beta integral:
$
\int_0^\infty
(1+s)^{-k}
s^{\sum a_i-1}\,ds
=
\frac{
\Gamma\left(\sum_{i=1}^r a_i\right)
\Gamma\left(k-\sum_{i=1}^r a_i\right)
}{
\Gamma(k)
},
$
which is valid since $k>\sum_i a_i$. The second integral does not depend on $k$.

Finally, using $a_i = 1/p_i$ and the functional equation $\Gamma(z+1)=z\Gamma(z)$, we get
\[
\frac{\theta_{k+1}(p_1,\dots,p_r)}{\theta_k(p_1,\dots,p_r)}
=
\frac{
\Gamma\left(k+1-\sum_{i=1}^r \frac{1}{p_i}\right)
}{
\Gamma(k+1)
}
\cdot
\frac{
\Gamma(k)
}{
\Gamma\left(k-\sum_{i=1}^r \frac{1}{p_i}\right)
}
=
\frac{
k-\sum_{i=1}^r \frac{1}{p_i}
}{k}.
\]

\end{proof}
\begin{proof}[Proof of Theorem \ref{main theorem}]
Put $m:=\min(n,r)$. Given a tuple of positive numbers $(t_1,\dots,t_r)$, we define the tuple of plurisubharmonic functions
$$\Phi_i = \log \frac{t_i}{1+\sum_{j=1}^r t_j} + p_i\varphi_i\;\;(i = 1,\dots,r)\;\;\;\text{and}
\;\;\; \Phi_{r+1} = \log \frac{1}{1+\sum_{j=1}^r t_j}.$$
These functions satisfy $\sum_{i=1}^{r+1} e^{\Phi_i} <1$. By Theorem \ref{sharp L^2-division theorem}, applied to the $r+1$ weights above with $q=\min(n,r)=m$, we have a tuple of holomorphic functions $(h_1,\dots,h_{r+1})$ such that
$$ \sum_{i=1}^{r+1} h_i = f\;\;\; \text{and}$$
$$\sum_{i=1}^r \int_{\Omega} \frac{1+\sum_{j=1}^r t_j}{t_i} \abs{h_i}^2 e^{-\psi-p_i\varphi_i} + \int_{\Omega} \left(1+\sum_{j=1}^r t_j\right)\abs{h_{r+1}}^2 e^{-\psi}
\le \int_{\Omega} \frac{\abs{f}^2 e^{-\psi}}{\left(\sum_{i=1}^r\frac{t_ie^{p_i\varphi_i}}{1+\sum_{j=1}^r t_j} + \frac{1}{1+\sum_{j=1}^r t_j}\right)^{m+1}}.$$
From this inequality, we have $h_i \in \Gamma(\Omega,\cI(p_i \varphi_i+\psi))$ $(i = 1,\dots,r)$, and $f-h_{r+1} \in \sum_{i=1}^r\Gamma(\Omega,\cI(p_i\varphi_i+\psi))$.
This implies that
$$\frac{c_{\varphi_1,\dots,\varphi_r,f,\psi}(p_1,\dots,p_r)}{\left(1+\sum_{i=1}^r t_i\right)^{m}} \le \int_{\Omega} \frac{\abs{f}^2 e^{-\psi}}{\left(\sum_{i=1}^r t_i e^{p_i\varphi_i} + 1\right)^{m+1}}.$$
Now, putting $t_i = x_i^{p_i}$ and integrating both sides of the inequality over $\mathbb{R}_{>0}^r$, we obtain
$$\theta_m(p_1,\dots,p_r)c_{\varphi_1,\dots,\varphi_r,f,\psi}(p_1,\dots,p_r)
\le \int_{\mathbb{R}_{>0}}\cdots\int_{\mathbb{R}_{>0}}\int_{\Omega} \frac{\abs{f}^2 e^{-\psi}}{\left(\sum_{i=1}^r x_i^{p_i} e^{p_i\varphi_i} + 1\right)^{m+1}}\,dx_1\cdots dx_r.$$
By Tonelli's theorem, the right-hand side equals
$\theta_{m+1}(p_1,\dots,p_r)\int_{\Omega}\abs{f}^2 e^{-\psi-\varphi_1-\cdots-\varphi_r}$.
Using the above lemma and the assumption $\sum_{i=1}^r 1/p_i<m$, we obtain the desired estimate.
\end{proof}

\begin{exam}\label{ex:constant-bidisc}
Let $\Omega=\Delta^2$ be the unit bidisc, let $\psi=0$, and set
$\varphi_1=\frac{1}{p_1}\log |z_1|^2$ and
$\varphi_2=\frac{1}{p_2}\log |z_2|^2$, where $p_1,p_2>1$.
Then $p_1\varphi_1=\log |z_1|^2$ and
$p_2\varphi_2=\log |z_2|^2$, so that
$\mathcal I(p_1\varphi_1)=(z_1)$ and
$\mathcal I(p_2\varphi_2)=(z_2)$.

Let $f=1$. The admissibility condition
$1-g\in z_1\mathcal O(\Delta^2)+z_2\mathcal O(\Delta^2)$
is equivalent to $g(0,0)=1$. Since the constant function is orthogonal in
$A^2(\Delta^2)$ to every holomorphic function vanishing at the origin,
the unique minimizer is $g\equiv1$. Hence
$c_{\varphi_1,\varphi_2,1,0}(p_1,p_2)=\pi^2$.

On the other hand,
$I=\int_{\Delta^2}|z_1|^{-2/p_1}|z_2|^{-2/p_2}\,d\lambda$.
Since
$\int_\Delta |z|^{-2/p}\,d\lambda=\frac{\pi p}{p-1}$,
we have
\[
I=\frac{\pi^2p_1p_2}{(p_1-1)(p_2-1)}.
\]
Consequently,
\[
\begin{aligned}
c_{\varphi_1,\varphi_2,1,0}(p_1,p_2)
&=\pi^2\\
&<
\left(1-\frac12\left(\frac1{p_1}+\frac1{p_2}\right)\right)
\frac{\pi^2p_1p_2}{(p_1-1)(p_2-1)}\\
&=
\left(1-\frac12\left(\frac1{p_1}+\frac1{p_2}\right)\right)I.
\end{aligned}
\]
Indeed, after division by $I$, the strict inequality is equivalent to
\[
\left(1-\frac1{p_1}\right)\left(1-\frac1{p_2}\right)
<
1-\frac12\left(\frac1{p_1}+\frac1{p_2}\right),
\]
and the difference between the right-hand side and the left-hand side is
$\frac{p_1+p_2-2}{2p_1p_2}>0$.
\end{exam}

Example~\ref{ex:constant-bidisc} shows that the inequality in Theorem~\ref{main theorem} can be strict already for $r=2$. This suggests that a better constant may exist. It would therefore be interesting to determine the optimal constant.

\section{Relations among $L^2$-division, $L^2$-extension, and strong openness}\label{relation}

The $L^2$-division theorem, the $L^2$-extension theorem, and the strong openness theorem are closely related. In this section, we explain that the $L^2$-division theorem provides a common source for them, in the sense that suitable sharp division estimates imply both the sharp $L^2$-extension theorem and the sharp effective openness theorem.

We first recall the following classical $L^2$-extension theorem. Let $\Omega \subset \C^n$ be a pseudoconvex domain, and $S$ be a smooth closed submanifold with $\codim S =k \ge 1$. Assume that there exist holomorphic functions $T_i$ $(i = 1,\dots,k)$ such that $$ S = \{T_1 = \dots = T_k = 0\}\;\;\;and$$
$$dT := dT_1 \wedge \dots \wedge dT_k \ne 0\;\;\; on\; S.$$
We define the measure on $S$ : $d\lambda_{S} = \frac{d\lambda}{\imaginary^{k^2}dT\wedge d\bar{T}}$
, where $d\lambda$ is the Lebesgue measure on $\C^n$ .
We also assume that there exists a negative plurisubharmonic function $\Psi$ on $\Omega$ such that $\Psi \in C^{\infty}(\Omega\setminus S)$ and $\Psi = k\log\sum\abs{T_i}^2 + B(x)$ where $B$ is continuous.
The following $L^2$-extension theorem is a central result in $L^2$ theory.
\begin{thm}\label{sharp L2 extension}
  There exists a constant \(C_k > 0\) depending only on the codimension \(k\) such that the following holds:
  For any plurisubharmonic function $\varphi$ on $\Omega$ and holomorphic function $f \in \cO(S)$, satisfying $$ \int_S  \abs{f}^2 e^{-\varphi-B} d\lambda_S < +\infty,$$ there exists a holomorphic function $F \in \cO(\Omega)$  satisfying $F\vert_S = f$ and $$\int_{\Omega}  \abs{F}^2 e^{-\varphi} \le C_k \int_S  \abs{f}^2 e^{-\varphi-B} d\lambda_S.$$
\end{thm}
A natural question is how small the constant $C_k$ in the $L^2$-extension theorem can be chosen. The optimal constant problem was first solved by B{\l}ocki \cite{blocki suita} in the case where $\Omega$ is a pseudoconvex domain and
\[
S=\Omega\cap\{z_n=0\}
\]
is a hyperplane section. Guan and Zhou \cite{GZ15} subsequently established the result in the general setting.

\begin{thm}
  Under the assumptions of the $L^2$-extension theorem, the constant $C_k$ can be taken to be \[ C_k=\frac{\pi^k}{k!}, \] the volume of the unit ball in $\mathbb{C}^k$. Moreover, this constant is optimal.
\end{thm}

The $L^2$-extension theorem has numerous applications in complex analysis and complex geometry. Here, we briefly discuss only those related to $L^2$-division and strong openness.

Ohsawa gave a proof of Skoda's original $L^2$-division theorem using an $L^2$-extension theorem \cite{O04}. The geometric idea behind his argument can be described as follows. Let
$
g:E\longrightarrow Q
$
be a surjective holomorphic bundle morphism. The division problem may be regarded as the problem of determining whether the induced map
$
\Gamma(X,E)\longrightarrow \Gamma(X,Q)
$
is surjective, together with an appropriate $L^2$-estimate for a lifting.

After projectivization, the quotient morphism induces a natural closed embedding
\[
\mathbb{P}(Q)\hookrightarrow \mathbb{P}(E).
\]
Under the quotient convention used here, sections of $E$ and $Q$ are identified with sections of the tautological line bundles
\[
\mathcal{O}_{\mathbb{P}(E)}(1)
\quad\text{and}\quad
\mathcal{O}_{\mathbb{P}(Q)}(1),
\]
respectively. The division problem can therefore be interpreted as an extension problem: one seeks to extend a section of
$
\mathcal{O}_{\mathbb{P}(Q)}(1)
$
to a section of
$
\mathcal{O}_{\mathbb{P}(E)}(1).
$
Ohsawa's argument applies an $L^2$-extension theorem to this projective-bundle setting and thereby obtains an $L^2$-division result.

The $L^2$-extension theorem also played an important role in the original proof of the strong openness conjecture by Guan and Zhou \cite{GZ openness}. Their proof proceeds by contradiction and combines an $L^2$-extension argument with a curve-selection lemma. Thus, both $L^2$-division and strong openness can be approached through $L^2$-extension techniques.

Conversely, Albesiano gave in \cite{A25} a new proof of the $L^2$-extension theorem using an $L^2$-division theorem. Inspired by this argument, the author proved in \cite{M.T} a sharp $L^2$-division theorem and derived from it an alternative proof of the sharp $L^2$-extension theorem. The same sharp division theorem also implies a sharp effective openness result. In this sense, the sharp $L^2$-division theorem provides a common source for both sharp extension and sharp effective openness.

This raises the following natural converse problem:
\textit{Does the sharp $L^2$-extension theorem imply the sharp $L^2$-division theorem?}
Ohsawa's construction suggests that such an implication should be sought by applying a sharp extension theorem to the embedding
\[
\mathbb{P}(Q)\hookrightarrow \mathbb{P}(E).
\]
However, the known extension argument in this projective-bundle setting does not appear to recover the sharp constant in the corresponding division theorem.
Thus, the converse problem may also be viewed as asking whether one can
establish a stronger form of the sharp $L^2$-extension theorem, with an
optimal estimate attained even in such a geometrically nontrivial setting.

\section*{AI usage statement}
ChatGPT (OpenAI) was used as an editorial aid to check minor language and typographical errors, consistency of notation, and bibliographic details. It did not contribute to the formulation of the main mathematical results. All mathematical arguments, citations, and final wording were reviewed and verified by the author.


\begin{thebibliography}{99}
\bibitem{A25}
R. Albesiano,
\newblock \emph{From division to extension},
\newblock arXiv:2503.01061.

\bibitem{B openness}
B. Berndtsson,
\newblock \emph{The openness conjecture for plurisubharmonic functions},
\newblock arXiv:1305.5781.

\bibitem{B projective}
B. Berndtsson,
\newblock \emph{The openness conjecture for projective manifolds},
\newblock arXiv:1305.0544.

\bibitem{blocki suita}
Z. B{\l}ocki,
\newblock \emph{Suita conjecture and the Ohsawa--Takegoshi extension theorem},
\newblock Invent. Math. \textbf{193} (2013), no.~1, 149--158.

\bibitem{Demailly lec}
J.-P. Demailly,
\newblock \emph{Analytic Methods in Algebraic Geometry},
\newblock Surveys of Modern Mathematics, vol.~1, International Press, Somerville, MA; Higher Education Press, Beijing, 2012.

\bibitem{DK}
J.-P. Demailly and J. Koll{\'a}r,
\newblock \emph{Semi-continuity of complex singularity exponents and Kähler--Einstein metrics on Fano orbifolds},
\newblock Ann. Sci. {\'E}c. Norm. Sup{\'e}r. (4) \textbf{34} (2001), no.~4, 525--556.

\bibitem{FJ}
C. Favre and M. Jonsson,
\newblock \emph{Valuations and multiplier ideals},
\newblock J. Amer. Math. Soc. \textbf{18} (2005), no.~3, 655--684.

\bibitem{saitoh}
Q. Guan,
\newblock \emph{A proof of Saitoh's conjecture for conjugate Hardy $H^2$ kernels},
\newblock J. Math. Soc. Japan \textbf{71} (2019), no.~4, 1173--1179.

\bibitem{GZ openness}
Q. Guan and X. Zhou,
\newblock \emph{A proof of Demailly's strong openness conjecture},
\newblock Ann. of Math. (2) \textbf{182} (2015), no.~2, 605--616.

\bibitem{GZ eff}
Q. Guan and X. Zhou,
\newblock \emph{Effectiveness of Demailly's strong openness conjecture and related problems},
\newblock Invent. Math. \textbf{202} (2015), no.~2, 635--676.

\bibitem{GZ sharp eff}
Q. Guan,
\newblock \emph{A sharp effectiveness result of Demailly's strong openness conjecture},
\newblock Adv. Math. \textbf{348} (2019), 51--80.

\bibitem{GZ15}
Q. Guan and X. Zhou,
\newblock \emph{A solution of an $L^2$ extension problem with an optimal estimate and applications},
\newblock Ann. of Math. (2) \textbf{181} (2015), no.~3, 1139--1208.

\bibitem{O04}
T. Ohsawa,
\newblock \emph{Generalization of a precise $L^2$ division theorem},
\newblock Adv. Stud. Pure Math. \textbf{42} (2004), 249--261.

\bibitem{M.T}
M. Takakura,
\newblock \emph{On the sharp $L^2$ estimates of Skoda division theorem},
\newblock arXiv:2505.05938.

\end{thebibliography}
\end{document}